\pdfoutput=1
\documentclass{amsart}
\usepackage{amsrefs}
\usepackage{hyperref}
\usepackage{amssymb}
\usepackage{graphicx,color}

\theoremstyle{plain}
\newtheorem{theorem}{Theorem}[section]

\newtheorem{lemma}[theorem]{Lemma}

\theoremstyle{definition}
\newtheorem{definition}{Definition}

\newtheorem{example}{Example}

\theoremstyle{remark}
\newtheorem{remark}{Remark}

\newcommand{\StTwo}{St_{(2)}}
\newcommand{\StOne}{St_{(1)}}

\newcommand{\ind}{\operatorname{ind}}
\newcommand{\sgn}{\operatorname{sgn}}

\title[Slice formula for surface invariants]
{An explicit slice formula for surface invariants via curve invariants}
\author{Noboru Ito}
\address{Department of Mathematics, Faculty of Engineering, Shinshu University, Wakasato 4-17-1, Nagano, Nagano 380-8553, Japan}
\email{nito@shinshu-u.ac.jp}
\author{Hiroki Mizuno}
\address{Department of Science and Technology, Graduate School of Medicine, Science and Technology, Shinshu University, Asahi 3-1-1, Matsumoto, Nagano 390-8626, Japan}
\email{22hs602j@shinshu-u.ac.jp}
\date{April 5, 2026}

\begin{document}
\begin{abstract}
We give an explicit slice formula for a surface invariant of generic immersions in $\mathbb{R}^3$, expressed in terms of curve invariants arising from planar slices.
Using a motion-picture viewpoint, we introduce differential measures that record local changes of the curve invariant $\StOne$ and the surface invariant $\StTwo$ across singular slice transitions.
Our main result shows that, for a quadruple-point event, if $j$ denotes the number of outward coorientations before the event, then the change of the surface invariant satisfies $d\StTwo = 2j - 4$.
This yields a computable and combinatorial description of the surface invariant via slice data.
In particular, the formula makes explicit the relation between curve-level invariants and finite-order invariants of surface immersions in the sense of Nowik.
\end{abstract}

\maketitle
\section{Introduction}\label{Intro}

The study of topological invariants of curves and surfaces has played a central role
in low-dimensional topology and singularity theory.
Among them, Arnold-type invariants, originally introduced for plane curves
\cite{Arnold1994Book}, provide powerful tools for describing how numerical data change
when a regular homotopy crosses codimension-one strata of the discriminant,
such as self-tangency and triple-point transitions.
On the surface side, Nowik developed a systematic framework of finite-order invariants
for surface immersions, classifying codimension-one singularity events
and defining universal order-one invariants whose jumps are prescribed
at these events \cite{Nowik2004, Nowik2006AM, Nowik2006PJM}.
Our results provide an explicit and computable slice-level relation between curve-type and surface-type invariants. To our knowledge, such a relation has not previously been formulated in the framework of finite-order invariants of surface immersions.

\medskip
\noindent
\textbf{Main results.}
In this paper we study a slice-theoretic relation between Arnold-type invariants
of plane curves and surface immersions in $\mathbb{R}^3$.
Using a motion-picture viewpoint, we introduce differential measures
that describe local changes of these invariants across singular slice transitions.
This description is independent of auxiliary choices and depends only on the local combinatorics of the configuration.
Our main result gives an explicit slice formula describing the change
of the surface strangeness invariant under a quadruple-point event.
More precisely, we show that this change is completely determined
by the number of outward coorientations in the local tetrahedral configuration.
This provides a combinatorial and computable description of the surface invariant
in terms of slice data and makes explicit the relation between
Arnold-type invariants of curves and surface invariants.
In particular, it clarifies the relation between classical Arnold-type invariants
and finite-order invariants of surface immersions in the sense of Nowik.

The construction developed in this paper can be viewed as a combinatorial mechanism
relating finite-order invariants of surface immersions to slice-wise behavior of plane curve invariants.
From this perspective, the slice formula may be regarded as a local-to-global principle,
analogous in spirit to constructions appearing in the theory of finite-order invariants.
We expect that this viewpoint can be extended to more general classes of invariants
or higher-dimensional settings.

For the curve side, we introduce a new basepoint-independent invariant $\StOne$
for oriented multi-component plane curves, defined purely in terms of coorientation data via the induced Alexander numbering.
This invariant serves as the $1$-dimensional counterpart of the surface invariant $\StTwo$
in our slice-theoretic formulation.

\medskip
\noindent
\textbf{Differential slice calculus.}
The key technical ingredient underlying our construction
is a notion of \emph{differential measure} that captures local differences
of Arnold-type invariants between adjacent slices of a surface.
The need for such a device arises from the fact that, while a surface immersion
may be locally regular as a $2$-immersion, the corresponding $1$-dimensional slices
typically undergo singular transitions.
In particular, triple-point crossings for plane curves correspond to
quadruple-point events for surfaces.

We introduce differential measures describing these local slice transitions
and establish explicit slice formulas relating the differential measures
for $1$-immersions and $2$-immersions (Theorems~\ref{thm:DiffOneTwo} and~\ref{thm:SliceF}).

Throughout the paper, we fix the sign convention so that the positive direction
of a $\textup{Q}$-event corresponds to an increase in the number of outward coorientations.

In the curve case, the differential measure associated with a Reidemeister move
of type~$\Omega_3$ is determined by the number of outward coorientations
of the vanishing triangle:
\[
d\StOne(\Omega_3)=2j-3,
\]
where $d\StOne$ denotes the change of the invariant $\StOne$ and
$j$ is the number of outward coorientations before the move.

Thus the differential measure is expressed directly in terms of
combinatorial data of the local configuration.

For surfaces, we prove that the change of $\StTwo$ under a quadruple-point event $\textup{Q}$ is likewise determined by the number of outward coorientations:
\[
d\StTwo(\textup{Q}) = 2j - 4,
\]
where $j$ denotes the number of outward coorientations of the local tetrahedral configuration
before the event.

In the case of $\textup{Q}^2$, the number $j$ of outward coorientations
does not change across the event.
Hence one has
\[
d\StTwo(\textup{Q}^2)=0.
\]

\medskip
\noindent
\textbf{Structure of the paper.}
Section~\ref{sec:def-not} reviews Arnold-type invariants for surfaces
and codimension-one singularity events $\textup{E}$, $\textup{H}$, $\textup{T}$,
and $\textup{Q}$.
Section~\ref{sec:diff} introduces differential measures
and proves slice formulas relating triple-point transitions of curves
and quadruple-point events of surfaces
(Theorems~\ref{thm:DiffOneTwo} and~\ref{thm:SliceF}).
Section~\ref{sec:example} provides examples and explicit computations.

\section{Definitions and notations}\label{sec:def-not}

\begin{definition}[immersion]
A smooth immersion of a closed oriented surface into $\mathbb R^3$ is said to be \emph{generic} if its image locally exhibits only transverse intersections of two or three smooth sheets.
\end{definition}

When no confusion is likely, we identify a generic immersion with its image.

\begin{definition}[event]
Two generic immersions $f$ and $g$ are said to be \emph{generically regularly homotopic} if they can be transformed by a finite sequence of diffeomorphisms and the four local singularity events $\textup{E}$, $\textup{H}$, $\textup{T}$, and $\textup{Q}$.
\end{definition}

\begin{itemize}
  \item[$\textup{(E)}$] elliptic tangency of two sheets (Figure~\ref{fig:Emove});
  \item[$\textup{(H)}$] hyperbolic tangency of two sheets (Figure~\ref{fig:Hmove});
  \item[$\textup{(T)}$] (for `triple'), tangency of the line of intersection of two sheets to another sheet (Figure~\ref{fig:Tmove});
  \item[$\textup{(Q)}$] (for `quadruple'), four sheets intersecting at the same point (Figure~\ref{fig:Qmove}).
\end{itemize}

\begin{figure}[htbp]
    \includegraphics[width=10cm]{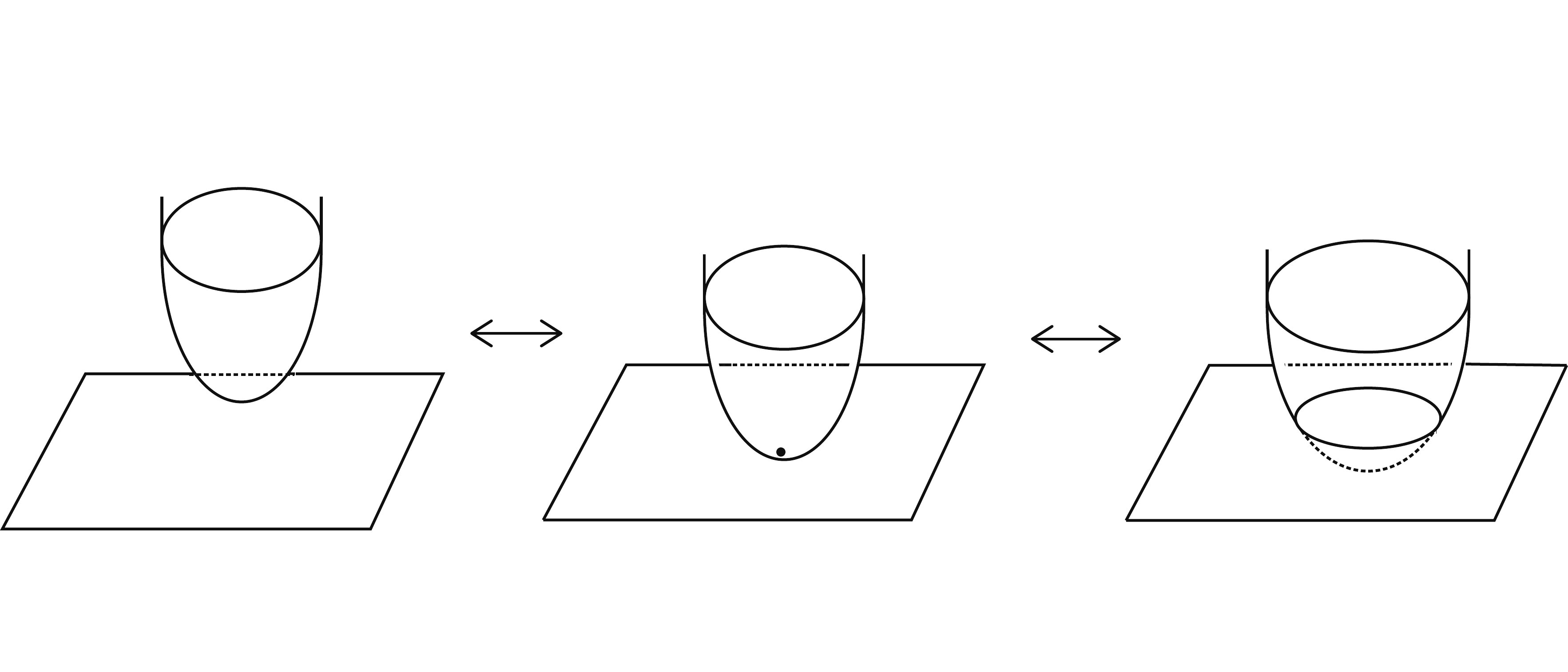}
    \vspace{-5mm}
    \caption{$\textup{(E)}$ : Elliptic tangency of two sheets.}
    \label{fig:Emove}
\end{figure}

\begin{figure}[htbp]
    \includegraphics[width=10cm]{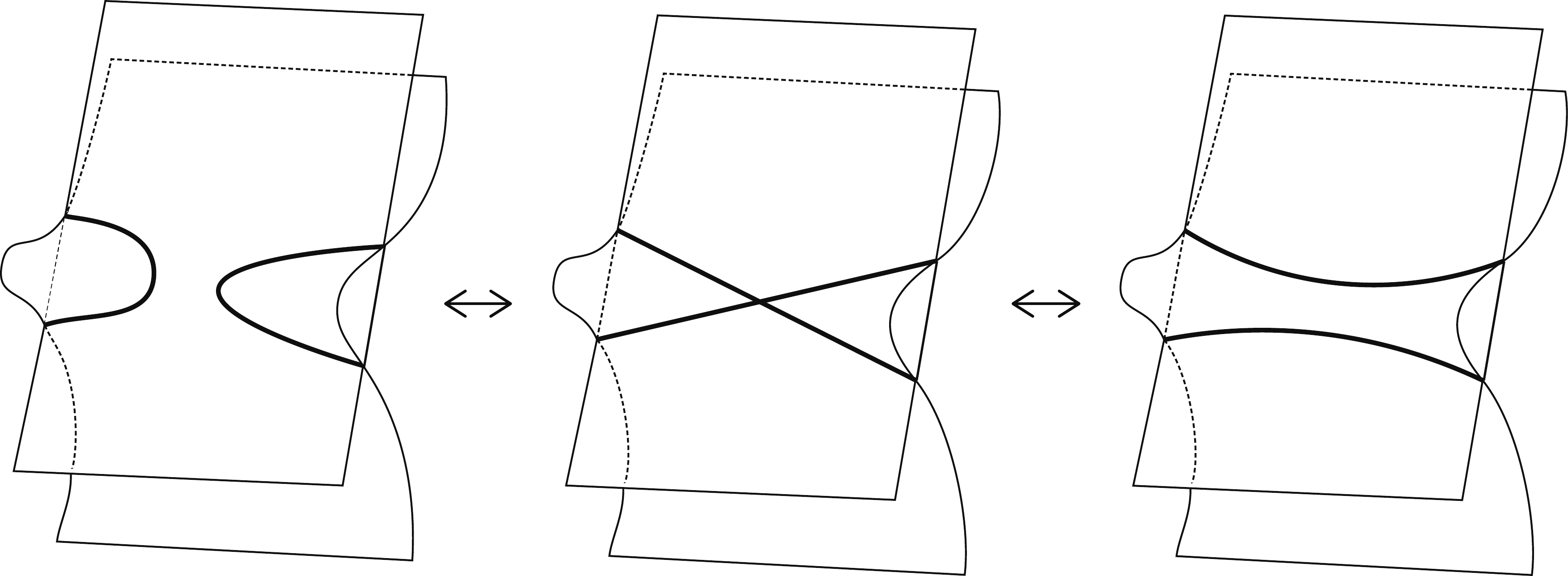}
    \caption{$\textup{(H)}$ : Hyperbolic tangency of two sheets.}
    \label{fig:Hmove}
\end{figure}

\begin{figure}[htbp]
    \includegraphics[width=10cm]{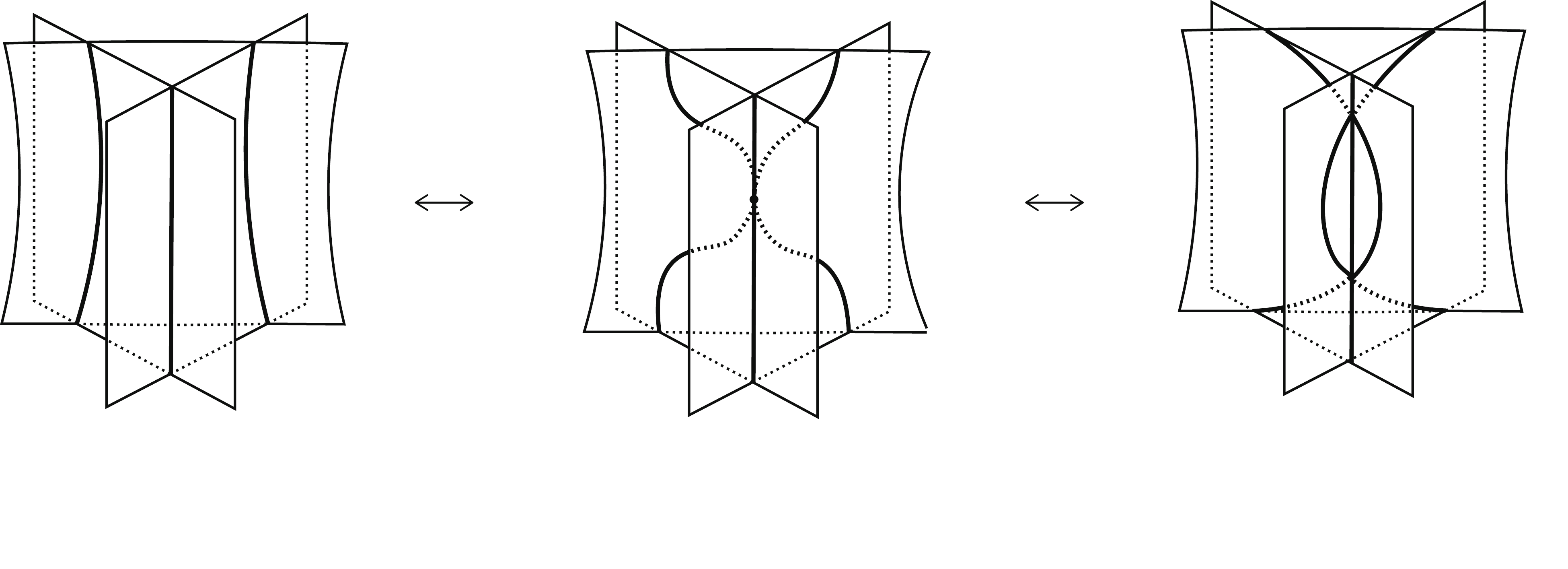}
    \vspace{-4mm}
    \caption{$\textup{(T)}$ : Tangency of the line of intersection of two sheets to another sheet.}
    \label{fig:Tmove}
\end{figure}

\begin{figure}[htbp]
    \includegraphics[width=10cm]{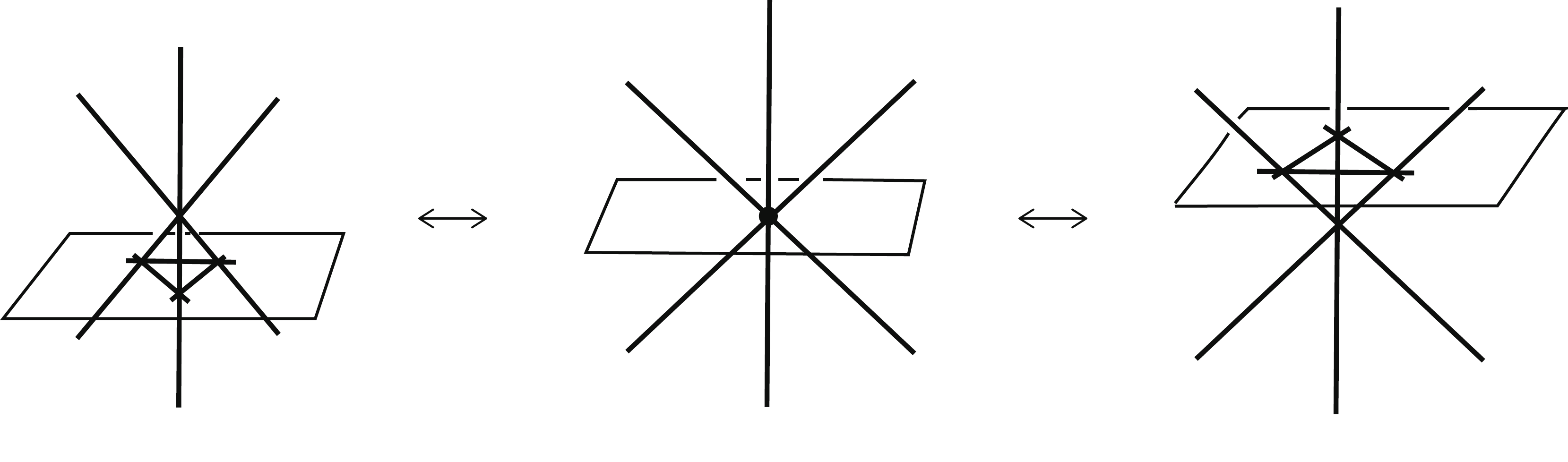}
    \caption{$\textup{(Q)}$ : A sheet crossing over a triple point formed by three other sheets. For the quadruple point, four sheets intersect at the same point.}
    \label{fig:Qmove}
\end{figure}

Let us fix an orientation of a closed surface $\Sigma$, and let $f:\Sigma \looparrowright \mathbb{R}^3$ be an immersion.
The orientation of $\Sigma$ induces an orientation on the image of $f$.
Moreover, if the target space $\mathbb{R}^3$ is oriented, then the image of $f$ is \emph{cooriented}.
Consequently, the coorientation is naturally established on the immersed surface as well.
Using this coorientation allows us to refine these singularity events.
The following content is based on \cite{Goryunov1997, Nowik2004}.
We recall \emph{codimension-one points}.

\begin{definition}[codimension-one points and their coorientation]
An event at a local codimension-one singularity, i.e., $\textup{E}$, $\textup{H}$, $\textup{T}$, or $\textup{Q}$, will be called a \emph{codimension-one point}.
By codimension-one point we also refer to the point in $\mathbb{R}^3$ where this event takes place.
By choosing an orientation of $\Sigma$ in $\mathbb{R}^3$, the events $\textup{E}$, $\textup{H}$, $\textup{T}$, and $\textup{Q}$ split into 12 types:
$\textup{E}^j$ ($j=0,1,2$), $\textup{H}^j$ ($j=+,-$), $\textup{T}^j$ ($j=0,1,2,3$), and $\textup{Q}^j$ ($j=2,3,4$).

The notion of outer and inner coorientation is determined globally
by the orientation of the surface $\Sigma$ in $\mathbb{R}^3$.  A part is called \emph{positive} if it has the outer coorientation.
We call the corresponding orientation also \emph{positive}
(resp.\ \emph{negative}) when it agrees with the outer
(resp.\ inner) orientation.

\begin{itemize}
\item[$\textup{E}^j$] For an event $\textup{E}$, let $j$ be the number of positive parts in the appearing sphere.
\item[$\textup{H}^j$] For an event $\textup{H}$, we set $\textup{H}=\textup{H}^+$ (resp.\ $\textup{H}=\textup{H}^-$) if the coorientations of the two sheets presenting the event $\textup{H}$ coincide (resp.\ do not coincide).
\item[$\textup{T}^j$] For an event $\textup{T}$, let $j$ be the number of positive parts in the appearing sphere.
\item[$\textup{Q}^j$] For an event $\textup{Q}$, let $4-j$ be the smaller number
of positive parts in the appearing sphere.
\end{itemize}
\end{definition}

Let $S:\Sigma \looparrowright \mathbb{R}^3$ be a generic immersion.
Then $\mathbb{R}^3$ is subdivided into cells.
The image decomposes into $2$-cells called \emph{sheets} (locally smooth pieces), $1$-cells called \emph{double lines} (pairwise intersections of sheets), and $0$-cells called \emph{triple points} (intersections of three sheets).
Furthermore, a $3$-cell in this subdivision is called a \emph{region}.

Using the coorientation of its image, we assign numbers to the regions; this assignment is called the \emph{Alexander numbering}.
The specific steps are as follows.
First, we assign $-\frac{3}{2}$ to the region containing the point at infinity.
Next, we increase or decrease the number along the coorientation of the image of $S$.
In this way, numbers are distributed to the regions.
The resulting map from regions to half-integers,
\[
\ind : \{\text{regions}\}\to \mathbb{Z}+\tfrac12
\]
is called the \emph{Alexander numbering}, and its values are called the \emph{indices of regions}.

\begin{definition}[index for triple points and $T(S)$]
Let $\Sigma$ be a closed oriented and cooriented surface, let $S$ be a generic immersion from $\Sigma$ to $\mathbb{R}^3$, and let $T(S)$ be the set of triple points of $S$.
By averaging the values of the Alexander numbering over the eight regions locally adjacent to a triple point, we define a function
\[
T(S)\to \mathbb{Z};\quad t\mapsto \ind(t).
\]
The value $\ind(t)$ is called the \emph{index} of a triple point $t$.
Similarly, we define the indices of sheets and double lines by averaging the Alexander numbering of the regions locally adjacent to them.
If $e$ is a double line and $f$ is a sheet of $S$, we denote their indices by $\ind(e)$ and $\ind(f)$, respectively.
\end{definition}

The definition of the index of a triple point will be used in the definition of the invariant $\StTwo$.

\begin{definition}[$\StTwo$]
Let $\Sigma$ be a closed oriented and cooriented surface, let $S:\Sigma \looparrowright \mathbb{R}^3$ be a generic immersion, and let $T(S)$ be the set of triple points of $S$.
Then
\[
\StTwo(S) = \sum_{t \in T(S)} \ind(t).
\]
\end{definition}

We now define the invariant $\StOne$ for oriented multi-component plane curves.
This invariant is independent of the choice of basepoint.

\begin{definition}[$\StOne$]\label{def:StOne}
Let $C$ be a generic immersion from $S^1 \sqcup S^1 \sqcup \cdots \sqcup S^1$ to $\mathbb{R}^2$.
We identify $C$ with its image when no confusion arises.

For an oriented generic multi-component plane curve, we assign to each double point an index defined as the average of the Alexander numberings of the adjacent regions.
We define
\[
\StOne(C):=\sum_{d \in D(C)} \ind(d), 
\]
where $D(C)$ denotes the set of double points of the image of $C$.  
\end{definition}

\section{Differential Measures and Slice Formulas}\label{sec:diff}

\begin{definition}[Morse surface immersion]
In order to investigate the relation between $\StTwo$ and $\StOne$, we consider cross-sections taken in neighborhoods of singular points on the surface.

Let $S:\Sigma \looparrowright \mathbb{R}^3$ be a generic immersion and place its image in a position so that the height function behaves like a Morse function.
A generic immersion arranged in this manner is called a \emph{Morse surface immersion}.

Let $z$ denote the height coordinate, and define the cross-section at level $z=s$ by
\[
A_s=S\cap\{z=s\}.
\]

Thus $\{A_s\}_{s\in\mathbb{R}}$ forms a one-parameter family of curves obtained by slicing the surface by the horizontal plane $z=s$.
We refer to this horizontal plane as the slice plane $\underline{\mathbb{R}^2}$.
\end{definition}
\begin{remark}[Independence of orientation of the slice curve]
The invariant $\StOne$ is defined using the Alexander numbering
induced by the coorientation of the surface,
and does not require an orientation of the slice curve $A_s$ itself.
In particular, the construction depends only on coorientation data
and local combinatorics of the configuration.
\end{remark}
\begin{definition}[index of slice plane]\label{def:indslice}
We describe the possible local index patterns arising on the slice plane near a vanishing triangle.
The indices at the $0$-cells are of one of the following two types:

\begin{itemize}
\item weak type
\[
(i,i,i+1)\longleftrightarrow(i,i+1,i+1),
\]

\item strong type
\[
(i,i,i)\longleftrightarrow(i+1,i+1,i+1).
\]
\end{itemize}

In both cases we denote the index by $i$.

Similarly, the indices at the $2$-cells of the vanishing triangle are of one of the following two types:

\begin{itemize}
\item weak type
\[
i\longleftrightarrow i+1,
\]

\item strong type
\[
i\longleftrightarrow i+3.
\]
\end{itemize}

Here all indices are computed on the slice plane.  
In particular, the indices of a vanishing triangle are computed on the slice $A_s$.
\end{definition}

\begin{definition}[differential measure $d\StOne$ for $\StOne$]

Let $\Omega_3$ be a Reidemeister move of type~$3$ for an oriented plane curve.
We define

\[
d\StOne(\Omega_3)
:=
\StOne(\text{after})-\StOne(\text{before}).
\]

\end{definition}

\begin{lemma}\label{lem:StOne}

Let $j$ be the number of outward coorientations of the vanishing triangle
before the move.
Then

\[
d\StOne(\Omega_3)=2j-3 .
\]

\end{lemma}

\begin{proof}

Let $j$ be the number of outward coorientations before the move.
A direct computation using the Alexander numbering on the slice plane
shows that the change of $\StOne$ is

\[
d\StOne=j-(3-j)=2j-3.
\]
Here $j$ denotes the number of outward edges of the vanishing triangle,
whereas $3-j$ denotes the number of inward edges.
These inward edges correspond to the directions in which the indices increase on the slice plane.    
\end{proof}

\begin{definition}[differential measure $d\StTwo$ for $\StTwo$]

We define

\[
d\StTwo
=
\StTwo(\text{after})-\StTwo(\text{before})
\]

for a quadruple-point event.

\end{definition}

We study local singular transitions of surface immersions through a slice plane.
Choose one of the sheets involved in the local deformation.
Its intersections with the other sheets form double lines,
which appear as plane curves on the slice.
From this viewpoint, a quadruple-point event corresponds to
a Reidemeister move $\Omega_3$ on the slice.

The following relation plays a central role.
It is not a definition but follows from a local comparison
of index contributions on the slice and in the ambient space.
We fix the sign convention so that the positive direction of a $\textup{Q}$-event
corresponds to an increase in the number of outward coorientations.
In the case of $\textup{Q}^2$, the number of outward coorientations does not change across the event.

\begin{theorem}\label{thm:DiffOneTwo}
Let $\underline{\mathbb{R}^2}$ be a slice plane taken immediately before
and after a quadruple-point event.

We fix the slice plane with its coorientation given by the positive direction of the height function.
Define
\[
\sgn(\underline{\mathbb{R}^2})=
\begin{cases}
-1 & \text{if the local transition is upward},\\
+1 & \text{if the local transition is downward}.
\end{cases}
\]

Then
\[
d\StTwo=d\StOne+\sgn(\underline{\mathbb{R}^2}).
\]
\end{theorem}

\begin{proof}
It suffices to consider quadruple-point events
for which the number of outward coorientations increases.

By symmetry, we may assume that the moving sheet passes the triple-point configuration
in the upward direction.

Viewing the configuration from above along the height direction of the moving sheet,
the induced slice transition is a Reidemeister move of type $\Omega_3$,
where the vanishing triangle is formed by the intersections of the moving sheet
with the other three sheets.

In this description, the number of outward coorientations of sheets
corresponds to the number of edges of the vanishing triangle
whose orientation agrees with the direction of increasing index
on the slice plane.
Thus the same parameter can be interpreted combinatorially both
for the surface and for the induced plane curve.

We examine the two cases for which the number of outward coorientations increases,
and the case $\textup{Q}^2$.

\smallskip
\noindent
 (1) Suppose that the number of outward coorientations before the event is $j=0$.
This is an event of type $\textup{Q}^{4-j}=\textup{Q}^4$.
The induced slice move is an $\Omega_3$-move in which the corresponding
vanishing-triangle data changes from $0$ to $3$, and one obtains
\[
d\StOne=-3.
\]

\smallskip
\noindent
 (2) Suppose that the number of outward coorientations before the event is $j=1$.
This is an event of type $\textup{Q}^{4-j}=\textup{Q}^3$.
The induced slice move is an $\Omega_3$-move in which the corresponding
vanishing-triangle data changes from $1$ to $2$, and one obtains
\[
d\StOne=-1.
\]

\smallskip
\noindent
(3) Suppose that the number of outward coorientations before the event is $j=2$.
This is an event of type $\textup{Q}^{4-j}=\textup{Q}^2$. 
The induced slice move is an $\Omega_3$-move in which the corresponding
vanishing-triangle data changes from $2$ to $1$, and one obtains
\[
d\StOne=+1.
\]

\smallskip
\noindent
In all cases this agrees with the formula
\[
d\StOne = 2j-3,
\]
where $j$ denotes the number of outward coorientations before the move.

Next, we compare this with the change of $\StTwo$.
Consider the local tetrahedral configuration
(Figure~\ref{fig:Qmove}).
Although the upward motion of the moving sheet increases
the number of sheets with outward coorientation by one,
the corresponding contribution to the differential measure is $-1$
due to the fixed coorientation of the slice plane and
the direction of the local transition.  

Therefore,
\[
\StTwo(\text{after})-\StTwo(\text{before})
= d\StOne - 1.
\]

Finally, by the definition of the sign of the slice plane,
this correction term coincides with $\sgn(\underline{\mathbb{R}^2})$.
Hence we obtain
\[
d\StTwo = d\StOne + \sgn(\underline{\mathbb{R}^2}),
\]
as claimed.
\end{proof}

\begin{remark}[Sign convention for the slice plane]
The sign of the slice plane is fixed so that
\[
d\StTwo = d\StOne + \sgn(\underline{\mathbb{R}^2})
\]
holds consistently with the orientation of the local tetrahedral model.
With the slice plane fixed, an upward passage of the moving sheet contributes $-1$.
\end{remark}

\begin{theorem}[Strangeness slice formula]\label{thm:SliceF}
Let $j$ be the number of outward coorientations of the sheets
before the event, so that the event is of type $\textup{Q}^{4-j}$.
Then
\[
d\StTwo = 2j-4.
\]
\end{theorem}

\begin{proof}
For a quadruple-point event the induced slice transition
is a Reidemeister move $\Omega_3$.
By Lemma~\ref{lem:StOne},
\[
d\StOne = 2j-3,
\]
where $j$ is the number of outward coorientations before the event.

Since the slice plane crosses the tetrahedral configuration once,
the resulting sign contribution is $-1$,
as determined by the local orientation convention on the slice.
Therefore
\[
d\StTwo=(2j-3)-1=2j-4.
\]
\end{proof}

\section{Table of Local Changes}\label{sec:table}

Table~\ref{StTwo^1} summarizes the local changes of the invariant $\StTwo$
under codimension-one events.
These values follow from the definitions of the indices
and the combinatorial structure of the singular configurations.

In Table~\ref{StTwo^1} we record the local contributions
to the invariant $\StTwo$ arising from different cells
in the tetrahedral decomposition of a neighborhood
of a codimension-one event.

More precisely, the notation
\[
\StTwo(S;\sigma^k)
\]
denotes the contribution coming from the $k$-dimensional
cells ($k=0,1,2,3$) in this local decomposition.
Here $\sigma^0$ corresponds to vertices (triple points),
$\sigma^1$ to double lines,
$\sigma^2$ to sheets,
and $\sigma^3$ to regions.

The invariant $\StTwo(S)$ itself is defined as the sum
of the indices of triple points (i.e.\ the $0$-cell
contribution), but the decomposition into
$\sigma^k$-terms is convenient for describing
local changes across codimension-one events.

\begin{table}[htbp]
  \centering
  \caption{Local contributions to the change of $\StTwo$
under codimension-one events.}\label{StTwo^1}
  \begin{tabular}{|l|c|c|c|c|}
    \hline
    & $\mathrm{St}_{(2)}(S;\sigma^{0})$ 
    & $\mathrm{St}_{(2)}(S;\sigma^{1})$
    & $\mathrm{St}_{(2)}(S;\sigma^{2})$
    & $\mathrm{St}_{(2)}(S;\sigma^{3})$ \\
    \hline
    $E^{0}$ & $0$ & $i+1$ & $2i+3$ & $i+2$ \\
    \hline
    $E^{1}$ & $0$ & $i$ & $2i$ & $i$ \\
    \hline
    $E^{2}$ & $0$ & $i-1$ & $2i-3$ & $i-2$ \\
    \hline
    $T^{0}$ & $2i+3$ & $3(2i+3)$ & $3(2i+3)$ & $2i+3$ \\
    \hline
    $T^{1}$ & $2i+1$ & $3(2i+1)$ & $3(2i+1)$ & $2i+1$ \\
    \hline
    $T^{2}$ & $2i-1$ & $3(2i-1)$ & $3(2i-1)$ & $2i-1$ \\
    \hline
    $T^{3}$ & $2i-3$ & $3(2i-3)$ & $3(2i-3)$ & $2i-3$ \\
    \hline
    $H^{+}$ & $0$ & $0$ & $0$ & $0$ \\
    \hline
    $H^{-}$ & $0$ & $0$ & $0$ & $-2$ \\
    \hline
    $Q^{4}$ & $-4$ & $-12$ & $-12$ & $-4$ \\
    \hline
    $Q^{3}$ & $-2$ & $-6$ & $-6$ & $-2$ \\
    \hline
    $Q^{2}$ & $0$ & $0$ & $0$ & $0$ \\
    \hline
  \end{tabular}
  \end{table}

\section{Example}\label{sec:example}
\begin{example}[Repeated $\textup{Q}^3$ events]\label{ex:Q3braid}
We describe a braid-like construction realizing repeated quadruple-point events of type $\textup{Q}^3$.

Take three sheets arranged so that they form a prism-like configuration
with a fixed triple-point structure in each slice.
Assume that exactly one of the three sheets has the opposite coorientation
to the other two.

Now let a fourth sheet pass transversely through this configuration.
Then exactly one of the four sheets has outward coorientation before the event,
so that the resulting event is of type $\textup{Q}^{4-j}$ with $j=1$,
i.e.\ of type $\textup{Q}^3$.

By Theorem~\ref{thm:SliceF}, each event contributes
\[
d\StTwo = 2\cdot 1 - 4 = -2,
\]
since $j=1$ before the event.
Reversing the direction of the motion yields $d\StTwo = 2$.

Repeating this construction produces arbitrarily many such
quadruple-point events.  
If one wishes to keep the surface connected,
the components can be joined by connected sums away from the local
configuration.
\end{example}

\begin{remark}[Range of the differential measure]\label{rem:range}
Since every local change of $\StTwo$ has the form $2j-4$,
the differential measure always takes values in $2\mathbb Z$.
On the other hand, Example~\ref{ex:Q3braid} realizes the minimal nonzero
value $|d\StTwo|=2$, and reversing the direction of the motion changes the sign.  

Therefore, by concatenating such local motions,
the total differential measure realizes every even integer.
Hence every even integer can be realized as the total differential measure.
\end{remark}

\bibliographystyle{plain}
\bibliography{SliceRef}

\end{document}